\documentclass[11pt]{amsart}
\usepackage{amssymb,color}
\usepackage{tikz}
\theoremstyle{plain}
\newtheorem{theorem}{Theorem}[section]

\newtheorem{remark}{Remark}[section]

\newtheorem{proposition}{Proposition}[section]

\theoremstyle{definition}

\theoremstyle{remark}

\allowdisplaybreaks

\long\def\symbolfootnote[#1]#2{\begingroup
	\def\thefootnote{\fnsymbol{footnote}}\footnote[#1]{#2}\endgroup}

\begin{document}

\title[The behavior of solutions of a weighted $(p,q)$-Laplacian equation]
	{The behavior of solutions of a parametric\\ weighted $(p,q)$-Laplacian equation}
	\author[D.D. Repov\v{s} and C. Vetro]{Du\v{s}an D. Repov\v{s} and Calogero Vetro}
	\address[D. D. Repov\v{s}]{Faculty of Education and Faculty of Mathematics and Physics,
		University of Ljubljana \& Institute of Mathematics, Physics and Mechanics,
		SI-1000, Ljubljana, Slovenia}
	\email{dusan.repovs@guest.arnes.si}
	\address[C. Vetro]{Department of Mathematics and Computer Science, University of Palermo, Via Archirafi 34, 90123, Palermo, Italy}
	\email{calogero.vetro@unipa.it}
	

	\date{}
	\maketitle

\begin{abstract}
	We study the behavior of solutions for the parametric equation
$$-\Delta_{p}^{a_1} u(z)-\Delta_{q}^{a_2} u(z)=\lambda |u(z)|^{q-2} u(z)+f(z,u(z))  \quad \mbox{in } \Omega,\, \lambda >0,$$
under Dirichlet condition, where	$\Omega \subseteq \mathbb{R}^N$ is a bounded domain with a $C^2$-boundary $\partial \Omega$,  $a_1,a_2 \in L^\infty(\Omega)$ with $a_1(z),a_2(z)>0$ for a.a. $z \in \Omega$, $p,q \in (1,\infty)$ and  $\Delta_{p}^{a_1},\Delta_{q}^{a_2}$ are weighted versions of $p$-Laplacian and $q$-Laplacian. We prove existence and nonexistence of nontrivial solutions, when $f(z,x)$  asymptotically as $x \to \pm \infty$ can be resonant. In the studied cases, we adopt a variational approach and use truncation and comparison techniques. When $\lambda$ is large, we establish the existence of at least three nontrivial smooth solutions with sign information and ordered. Moreover, the critical parameter value is determined in terms of the spectrum of one of the differential operators.
\end{abstract}

\maketitle

\section{Introduction}\label{sec:1}
Our goal here is to investigate the existence and nonexistence of nontrivial smooth solutions for  the following parametric Dirichlet problem
\begin{equation}\label{eq0}\tag{$P_\lambda$} 
	\begin{cases}-\Delta_{p}^{a_1} u(z)-\Delta_{q}^{a_2} u(z)=\lambda |u(z)|^{q-2} u(z)+f(z,u(z))  & \mbox{in } \Omega,\\ u \Big|_{\partial \Omega} =0, \, 1<q<p, \,  \lambda>0,& \end{cases}
\end{equation}
where $\Omega \subseteq \mathbb{R}^N$ is a bounded domain with a $C^2$-boundary $\partial \Omega$. Given $r \in (1,\infty)$ and $a \in L^\infty(\Omega)$ with $a(z)>0$ for a.a. $z \in \Omega$, by   $\Delta_{r}^a$ we mean the weighted $r$-Laplacian of the form
$\Delta_{r}^a u= \mbox{div }(a(z)|\nabla u|^{r-2}\nabla u)$ for all $u \in W_0^{1,r}(\Omega)$. Thus, \eqref{eq0} is driven by the operator $-\Delta_{p}^{a_1} -\Delta_{q}^{a_2}$, whose weights $a_1,a_2$ are Lipschitz continuous, positive and bounded away from zero. These conditions imply that the integrand corresponding to this differential operator, exhibits balanced growth. However, the fact that the two weights are different, does not allow the use of the nonlinear strong maximum principle (see Pucci and Serrin \cite{Ref21}, pp. 111, 120). Instead we use a recent result of Papageorgiou et al. \cite{Ref18}, together with an additional comparison argument, which allows us to conclude that the constant sign solutions of the problem satisfy the nonlinear Hopf's lemma. The right-side of \eqref{eq0} is the sum of the power term $\lambda |x|^{q-2} x$ and of the Carath\'eodory function $f(z,x)$. The $\lambda$-parametric term is $(p-1)$-sublinear (recall that $q<p$), and $f(z,x)$  is $(p-1)$-linear as $x \to \pm \infty$  and can be resonant with respect to the first eigenvalue of  ($-\Delta_p^{a_1},W_0^{1,p}(\Omega)$). We mention that the power of the parametric term (namely $q$) is the same with the exponent of the second differential operator $-\Delta_q^{a_2}$. This distinguishes \eqref{eq0} from problems with concave terms, where the power of the parametric term is strictly less than the exponents of all the differential operators in the left-side. Such concave problems, were studied recently by Gasi\'{n}ski and Papageorgiou \cite{Ref5},  Gasi\'{n}ski et al. \cite{Ref6} ($p$-equations),  Marano et al. \cite{Ref12}, Papageorgiou and \ Scapellato \cite{Ref17ter}, Papageorgiou and Zhang \cite{Ref20} ($(p,2)$-equations),  Papageorgiou et al. \cite{Ref17} (anisotropic equations), Papageorgiou and Winkert  \cite{Ref19}, Papageorgiou and Zhang \cite{Ref20b,Ref20c} ($(p,q)$-equations) and Papageorgiou et al. \cite{Ref17bis} (nonhomogeneous Robin problems).

Let $\widehat{\lambda}_1(q,a_2)>0$ be the principal eigenvalue of ($-\Delta_q^{a_2},W_0^{1,q}(\Omega)$).  Using variational tools from the critical point theory, truncation and comparison methods, then   \eqref{eq0} (for all $\lambda>\widehat{\lambda}_1(q,a_2)$) admits at least three nontrivial smooth solutions (positive, negative, nodal). Moreover, under an additional mild regularity for $f(z,\cdot)$, we get that   \eqref{eq0} (for all $\lambda<\widehat{\lambda}_1(q,a_2)$) has no nontrivial   solutions.

\section{Preliminaries}\label{sec:2}

A crucial point is to
establish the appropriate spaces, where carrying out the study. Here,  \eqref{eq0} is analyzed in  $ W^{1,p}_0(\Omega)$ (namely, Sobolev space) and in $C_0^1(\overline{\Omega})=\{u \in C^1(\overline{\Omega})  :   u \big|_{\partial \Omega}=0\}$ (classical Banach space). Additionally, $\|\cdot\|$ means the norm of $ W^{1,p}_0(\Omega)$ with
$$\|u \| = \|\nabla u\|_{p}\quad \mbox{for all } u \in W_0^{1,p}(\Omega)\quad \mbox{(by Poincar\'e inequality).}$$

$C_0^1(\overline{\Omega})$ is  ordered, with positive (order) cone 
$C_+=\left\{u \in C_0^1(\overline{\Omega})   :  u(z) \geq 0 \mbox{ for all } z \in \overline{\Omega}\right\}$. Now, $C_+$ has the nonempty interior
$$\mbox{int }C_+=\left\{u \in C_+   :  u(z) > 0  \mbox{ for all } z \in  \Omega , \quad \frac{\partial u}{\partial n} \Big|_{\partial \Omega}<0\right\},$$
with $n(\cdot)$ being the outward unit normal on $\partial \Omega$.  Let $r \in (1,\infty)$ and $ a\in C^{0,1}(\overline{\Omega})$ (that is, $a(\cdot)$ is Lipschitz continuous on $\overline{\Omega}$) with $a(z)\geq \widehat{c}_0> 0$ for all $z \in \overline{\Omega}$.

By $A_{r}^a: W_0^{1,r}(\Omega) \to  W^{-1,r^\prime }(\Omega)=W_0^{1,r}(\Omega)^\ast$ ($\frac{1}{r}+\frac{1}{r^\prime}=1$), we denote the  operator 
$$\langle A_{r}^a(u),h \rangle = \int_\Omega a(z)|\nabla u|^{r-2}(\nabla u, \nabla h)_{\mathbb{R}^N}dz \quad \mbox{for all }  u,h \in W_0^{1,r}(\Omega).$$

We recall some features of $A_{r}^a(\cdot)$ as follows:
\begin{itemize}
	\item  $A_{r}^a(\cdot)$ is bounded and continuous;
	\item $A_{r}^a(\cdot)$ is strictly monotone, and hence maximal monotone;
	\item $A_{r}^a(\cdot)$ is of type $(S)_+$. It means that, if $u_n  \xrightarrow{w} u $ in $W_0^{1,r}(\Omega)$ and $\limsup\limits_{n \to \infty} \langle A_{r}^a(u_n), u_n - u \rangle \leq 0$, then $u_n \to u$ in $W_0^{1,r}(\Omega).$ 
\end{itemize}

Given the eigenvalue problem 
$$ -\Delta_r^a u(z) =\widehat{\lambda}|u(z)|^{r-2}u(z) \mbox{ in $\Omega$, } \, u \Big|_{\partial \Omega}=0, 
$$
we say that $\widehat{\lambda} \in \mathbb{R}$ is an eigenvalue of $(-\Delta_r^a,W_0^{1,r}(\Omega))$, if the above problem admits a nontrivial solution $\widehat{u} \in W_0^{1,r}(\Omega)$ (namely,  eigenfunction of $\widehat{\lambda}$). There is a smallest eigenvalue $\widehat{\lambda}_1(r,a)>0$. Indeed, consider
\begin{align}
	0 \leq \widehat{\lambda}_1(r,a) & = \inf \left[\frac{\int_\Omega a(z)|\nabla u|^r dz}{\|u\|_r^r} : u  \in W_0^{1,r}(\Omega), \, u \neq 0  \right] \nonumber \\ & = \inf \left[ \int_\Omega a(z)|\nabla u|^r dz : u  \in W_0^{1,r}(\Omega), \, \|u\|_r=1  \right]. \label{eq1}	
\end{align}

We claim that the infimum in \eqref{eq1} is attained. To see this consider a sequence $\{u_n\}_{n \in \mathbb{N}} \subseteq W_0^{1,r}(\Omega)$ satisfying
$\|u_n\|_r=1$ for all $n \in \mathbb{N}$, and $\int_\Omega a(z)|\nabla u_n|^r dz \, \downarrow \, \widehat{\lambda}_1(r,a)$. From the boundedness of $\{u_n\}_{n \in \mathbb{N}} \subseteq W_0^{1,r}(\Omega)$, it is possible to suppose
\begin{equation}
	\label{eq2} u_n \xrightarrow{w} \widehat{u}_1 \mbox{ in $W_0^{1,r}(\Omega)$, $u_n \to \widehat{u}_1$ in $L^r(\Omega)$.}
\end{equation}

On account of our hypothesis on the weight $a(\cdot)$, on $L^r(\Omega,\mathbb{R}^N)$ $y \to \left[  \int_\Omega a(z)|y|^rdz\right]^{1/r}$ is an equivalent norm. From \eqref{eq2}, since the norm  (in Banach space)  is weakly lower semicontinuous, also using the Lagrange multiplier rule (Papageorgiou and  Kyritsi-Yiallourou \cite{Ref14}, p. 76) and the nonlinear regularity theory, after standard calculations we get $\widehat{u}_1  \in  C_0^1(\overline{\Omega})\setminus \{0\}$. Additionally, it is clear from \eqref{eq1} that we may assume that $\widehat{u}_1  \in  C_+\setminus \{0\}$ (just replace $\widehat{u}_1$ by $|\widehat{u}_1|$). Then the nonlinear Hopf's lemma (Pucci and Serrin \cite{Ref21}, pp. 111, 120), gives us $\widehat{u}_1=\widehat{u}_1(r,a)\in {\rm int \,}C_+$. From Jaros \cite[Theorem 3.3]{Ref9}, we know that 
$\widehat{\lambda}_1(r,a) \mbox{ is simple,}$
i.e., if $\widehat{u}_1,\widehat{v}_1$ are eigenfunctions corresponding to $\widehat{\lambda}_1$, then $\widehat{u}_1=\vartheta \,\widehat{v}_1$ for certain $\vartheta \in \mathbb{R} \setminus \{0\}$. Also $\widehat{\lambda}_1(r,a)>0$ is isolated in the spectrum $\sigma(r,a)$ of $(-\Delta_r^a,W_0^{1,r}(\Omega))$.  For this purpose, let us consider eigenvalues  $\{\widehat{\lambda}_n \}_{n \in \mathbb{N}}\subseteq \sigma(r,a)$ satisfying
$\widehat{\lambda}_1(r,a)<\widehat{\lambda}_n$ for all $n \in \mathbb{N}$, and $\widehat{\lambda}_n \, \downarrow \, \widehat{\lambda}_1(r,a).$
So, we can find $\widehat{u}_n \in W_0^{1,r}(\Omega)$, $\widehat{u}_n \neq 0$ such that 
$$ -\Delta_r^a \widehat{u}_n =\widehat{\lambda}_n|\widehat{u}_n|^{r-2}\widehat{u}_n \mbox{ in $\Omega$, } \, \widehat{u}_n\Big|_{\partial \Omega}=0, \, n \in \mathbb{N}. 
$$

By homogeneity we can always assume that $\|\widehat{u}_n\|_r=1$ for all  $n \in \mathbb{N}$. The nonlinear regularity theory (see Lieberman \cite{Ref11}), implies that there exist $\alpha \in (0,1)$ and $c_0>0$ such that
\begin{equation}
	\label{eq3} \widehat{u}_n \in C_0^{1,\alpha}(\overline{\Omega}), \, \|\widehat{u}_n\|_{C_0^{1,\alpha}(\overline{\Omega})}\leq c_0 \quad \mbox{for all } n \in \mathbb{N}.
\end{equation}

The compact embedding $C_0^{1,\alpha}(\overline{\Omega}) \hookrightarrow C_0^{1}(\overline{\Omega})$  and \eqref{eq3}, ensure one can suppose
\begin{align*}
	& u_n \to \widetilde{u} \mbox{ in }C_0^{1}(\overline{\Omega}), \, \|\widetilde{u}\|_r=1,\\
	\Rightarrow \quad & 	-\Delta_r^a \widetilde{u} =\widehat{\lambda}_1(r,a)|\widetilde{u}|^{r-2}\widetilde{u} \mbox{ in $\Omega$, } \, \widetilde{u}\Big|_{\partial \Omega}=0, \\
	\Rightarrow \quad & \widetilde{u}= \vartheta \, \widehat{u}_1 \in {\rm int \,}C_+\quad \mbox{for some $\vartheta>0$,}
\end{align*}
and hence $\widehat{u}_n \in {\rm int \,}C_+$  for all $n \geq n_0$,
which leads to contradiction with Jaros \cite[Corollary 3.2]{Ref9}. This proves that $\widehat{\lambda}_1(r,a)>0$ is isolated. The Ljusternik-Schnirelmann minimax scheme (see, for example, Gasi\'{n}ski and  Papageorgiou \cite{Ref3}), ensures a whole strictly increasing sequence of distinct eigenvalues $\{\widehat{\lambda}_n\}_{n \in \mathbb{N}}$ such that $\widehat{\lambda}_n \to +\infty$. If $r=2$, then these eigenvalues exhaust the spectrum. If $r \neq 2$, then it is not known if the LS-eigenvalues fully describe $\sigma(r,a)$. Moreover, every $\widehat{\lambda}\in \sigma(r,a)\setminus \{ \widehat{\lambda}_1(r,a)\}$ has eigenfunctions which are nodal functions (that is, sign-changing functions), see again Jaros \cite[Corollary 3.2]{Ref9}. We can easily check that $\sigma(r,a)\subseteq [\widehat{\lambda}_1(r,a),+\infty)$ is  closed. So, we can define the second eigenvalue of $(-\Delta_r^a,W_0^{1,r}(\Omega))$ by 
$$\widehat{\lambda}_2(r,a)=\inf [\widehat{\lambda} \in \sigma(r,a) : \widehat{\lambda}_1(r,a)<\widehat{\lambda}].$$

Reasoning as in Cuesta et al. \cite{Ref1}, one can show that $\widehat{\lambda}_2(r,a)$ corresponds to the second LS-eigenvalue and 
\begin{equation}
	\label{eq4} \widehat{\lambda}_2(r,a)=\inf\limits_{\widehat{\gamma} \in \widehat{\Gamma}}\max\limits_{-1 \leq t \leq 1}\int_\Omega a(z) |\nabla \widehat{\gamma}(t)|^rdz, 
\end{equation}
where $\widehat{\Gamma}=\{\widehat{\gamma} \in C([-1,1],M)  :  \widehat{\gamma}(-1)=-\widehat{u}_1(r,a), \,\widehat{\gamma}(1)=\widehat{u}_1(r,a) \}$ with $M= W_0^{1,r}(\Omega) \cap \partial B_1^{L^r}$ ($\partial B_1^{L^r}=\{u \in L^r(\Omega)   :   \|u\|_r=1\}$) and $\widehat{u}_1(r,a)$ is the positive, $L^r$-normalized eigenfunction (i.e., $\|\widehat{u}_1(r,a)\|_r=1$)  corresponding to $\widehat{\lambda}_1(r,a)>0$. Recall that $\widehat{u}_1=\widehat{u}_1(r,a)\in {\rm int \, }C_+$.

The above features lead to the following proposition.

\begin{proposition}
	\label{prop2} Let $\eta \in L^\infty(\Omega)$, $\eta(z)\leq \widehat{\lambda}_1(r,a)$ for a.a. $z \in \Omega$ and the inequality be strict on a set of positive Lebesgue measure. Then, $\int_\Omega a(z)|\nabla u|^p dz- \int_\Omega \eta(z) |u|^p dz \geq \widehat{c}\,  \|\nabla u\|^p$ for some $\widehat{c}>0$, all $u \in W_0^{1,p}(\Omega)$.
\end{proposition}

If $u :\Omega \to  \mathbb{R}$ is measurable, let $u^\pm(z) =\max \{\pm u(z),0\}$ for all $z \in \Omega$. If $u \in W^{1,p}_0(\Omega)$,  then $u^\pm \in W_0^{1,p}(\Omega)$ and $u=u^+-u^-$, $|u|=u^++u^-$. Also, if $u,v : \Omega \to  \mathbb{R}$ are measurable with $u(z)\leq v(z)$ for a.a. $z \in \Omega$, then we set:
$$[u,v]=\{h \in W_0^{1,p}(\Omega) :  u(z)\leq h(z)\leq  v(z)  \mbox{ for a.a. $z \in \Omega$}\}.$$

Now, ${\rm int}_{C^1_0(\overline{\Omega})}[u,v]$ means the interior in $C^1_0(\overline{\Omega})$ of $[u,v] \cap C^1_0(\overline{\Omega})$. For a Banach space $X$  and $\varphi \in C^1(X)$, let
$K_\varphi=\{u \in X   :  \varphi^\prime(u)=0\}$ (namely, critical set of $\varphi$). For $c \in \mathbb{R}$, let
$\varphi^c=\{u \in X : \varphi(u)\leq c\}$, $ K_\varphi^c=\{u \in K_\varphi : \varphi(u)= c\}$.

For a measurable function $g:\Omega \to \mathbb{R}$, then $0 \preceq g$ if and only if for every $K\subseteq \Omega$ compact, one has
$0<c_K \leq g(z)$ for a.a. $z \in K$.
When $g \in C(\Omega)$ and $g(z)>0$ for all $z \in \Omega$, clearly $0 \preceq g$.

In the study of \eqref{eq0}, we use the assumption $H_0$ stated as follows:

\medskip
\noindent $H_0$: $a_1,a_2 \in C^{0,1}(\overline{\Omega})$ and $0<c_1\leq a_1(z),a_2(z)$ for all $z \in \overline{\Omega}$. 

\begin{remark}
	If \ $\widehat{a}(z,y)=a_1(z)|y|^{p-2}y+ a_2(z)|y|^{q-2}y$ for all $(z,y)\in \Omega \times \mathbb{R}^N$, then we see that ${\rm div \,}a(z,\nabla u)=\Delta_p^{a_1}u +\Delta_q^{a_2}u$ for all $u \in W_0^{1,p}(\Omega)$. The primitive of \ $\widehat{a}(z,y)$ is the function $\widehat{G}(z,y)=\frac{a_1(z)}{p}|y|^p+ \frac{a_2(z)}{q}|y|^q$  for all  $(z,y)\in \Omega \times \mathbb{R}^N$. On account of $H_0$, we see that $\widehat{G}(\cdot,\cdot)$ exhibits balanced growth, namely
	$$\frac{c_1}{p}|y|^p\leq \widehat{G}(z,y)\leq c_2[1+|y|^p] \quad \mbox{for some $c_2>0$ and all $(z,y)\in \Omega \times \mathbb{R}^N$.}$$
\end{remark}

We also consider the following set of assumptions on the data:

\medskip
\noindent $H_1$: $f:\Omega \times \mathbb{R} \to \mathbb{R}$ is Carath\'eodory with $f(z,0)=0$ for a.a. $z \in \Omega$, and 
\begin{itemize}
	\item[$(i)$] for every $\rho>0$, there exists $a_\rho \in L^\infty(\Omega)$ with  $|f (z,x)| \leq a_\rho(z)$ for a.a. $z \in \Omega$,  all $|x| \leq \rho$;\\
	\item[$(ii)$]  $\limsup\limits_{x \to \pm \infty} \dfrac{f(z,x)}{|x|^{p-2}x}\leq\widehat{\lambda}_1(p,a_1)$ uniformly for a.a. $z \in \Omega$;\\
	\item[$(iii)$] 
	If $F(z,x)=\int_0^x f(z,s)ds$,	 there is $\tau \in (q,p)$ with
	$ \lim_{x \to \pm \infty} \dfrac{f(z,x)x-pF(z,x)}{|x|^{\tau}}=+\infty$   uniformly for a.a. $z \in \Omega$;\\
	\item[$(iv)$] $ \lim\limits_{x \to 0} \dfrac{f(z,x)}{|x|^{q-2}x}=0$   uniformly for a.a. $z \in \Omega$;\\
	\item[$(v)$] for every $s >0$, there exists $m_s>0$ with $m_s \leq f(z,x)x$ for a.a. $z \in \Omega$, all $|x|\geq s$.
\end{itemize}

\begin{remark}
	According to $H_1(ii)$,  we can have resonance of \eqref{eq0}  with respect to  $\widehat{\lambda}_1(p,a_1)>0$. By the proof of Proposition \ref{prop3}, we will see that this phenomenon originates from the left of $\widehat{\lambda}_1(p,a_1)$ in the sense that 
	$$\lim_{x \to \pm \infty}[\widehat{\lambda}_1(p,a_1)|x|^p-pF(z,x)]=+\infty \quad \mbox{uniformly for a.a. $z \in \Omega$.}$$
	
	We stress that this ensures the coercivity of the corresponding energy functional. Therefore, we can use classical tools of the calculus of variations. Assumption $H_1(iv)$ does not permit the presence of a concave term and this changes the geometry of our problem compared to those of the ``concave'' works mentioned in the Introduction. Finally we mention that assumptions $H_1$ imply that
	\begin{equation}
		\label{eq5}|f(z,x)|\leq a(z)[1+|x|^{p-1}]\quad \mbox{for a.a. $z \in \Omega$, all $x \in \mathbb{R}$, $a \in L^\infty(\Omega)_+$.}
	\end{equation}
	
\end{remark}

When $q=2$, we improve our conclusion about the nodal solution, provided we add a perturbed monotonicity assumption for $f(z,\cdot)$, as follows

\medskip

\noindent $H_1^\prime$:  $H_1$ hold (with $q=2$) and 
\begin{itemize}
	\item[$(vi)$] for every $\rho>0$, there exists $\widehat{\xi}_\rho>0$ such that for a.a. $z \in \Omega$, the function
	$x \to f(z,x)+  \widehat{\xi}_\rho|x|^{p-2}x$
	is nondecreasing on $[-\rho,\rho]$.
\end{itemize}

Finally, we can have a nonexistence result for \eqref{eq0} provided we add a growth restriction for $f(z,\cdot)$, as follows

\medskip

\noindent $H_1^{\prime\prime}$:   $H_1$ hold and 
\begin{itemize}
	\item[$(vi)$] $f(z,x)x\leq \widehat{\lambda}_1(p,a_1)|x|^p$ for a.a. $z \in \Omega$, all $x \in \mathbb{R}$.
\end{itemize}

\section{Positive and negative solutions}\label{sec:3}

The existence of positive and negative solutions for \eqref{eq0} is established in the case $\lambda > \widehat{\lambda}_1(q,a_2)$. We obtain smallest positive  and biggest negative solutions. These solutions of \eqref{eq0} (namely, extremal constant sign solutions) play a crucial role in Section \ref{sec:4} to generate a nodal solution.

\begin{proposition}
	\label{prop3} Let $H_0$, $H_1$ be satisfied, and $\lambda>\widehat{\lambda}_1(q,a_2)$. Then \eqref{eq0} admits  solutions $u_\lambda \in {\rm int \,}C_+$, $v_\lambda \in -{\rm int \,}C_+$.
\end{proposition}
\begin{proof}
	
	Let $\varphi_\lambda^+ : W_0^{1,p}(\Omega) \to \mathbb{R}$ be  a  $C^1$-functional given as
	$$ \varphi_\lambda^+(u)= \frac{1}{p}\int_\Omega a_1(z)|\nabla u|^{p}dz + \frac{1}{q}\int_\Omega a_2(z) |\nabla u|^q dz - \frac{\lambda}{q}\|u^+\|_q^q- \int_\Omega F(z,u^+)dz $$
	for all $ u \in W^{1,p}_0(\Omega)$.
	We discuss the properties of $\varphi_\lambda^+(\cdot)$ to obtain a positive solution of \eqref{eq0}. As already mentioned the coercivity of functionals is a crucial key to apply the direct methods of calculus of variations.
	
	\medskip
	
	\underline{Claim:} $\varphi_\lambda^+(\cdot)$ is coercive.\\
	
	Arguing by contradiction, suppose that there is $\{u_n\}_{n \in \mathbb{N}}\subseteq W_0^{1,p}(\Omega)$ satisfying
	\begin{align}
		\label{eq6} & \varphi_\lambda^+(u_n)\leq c_3 \quad \mbox{for some $c_3>0$, all $n \in \mathbb{N}$,}\\ \label{eq7} & \|u_n\| \to \infty \quad \mbox{as $n \to \infty$.}
	\end{align}
	
	If $\{u_n^+\}_{n \in \mathbb{N}}\subseteq W_0^{1,p}(\Omega)$ is bounded, then from \eqref{eq6} we deduce the boundedness of $\{u_n^-\}_{n \in \mathbb{N}}\subseteq W_0^{1,p}(\Omega)$. Consequently, we get the boundedness of $\{u_n\}_{n \in \mathbb{N}}\subseteq W_0^{1,p}(\Omega)$, which contradicts \eqref{eq7}. Therefore, one can suppose
	\begin{equation}
		\label{eq8}  \|u_n^+\| \to \infty \quad \mbox{as $n \to \infty$.}
	\end{equation}
	
	We set $y_n =\dfrac{u_n^+}{\|u_n^+\|}$, $n \in \mathbb{N}$. Then $\|y_n\|=1$, $y_n \geq0$ for all $n \in \mathbb{N}$. So, we suppose 
	\begin{equation}
		\label{eq9} y_n \xrightarrow{w} y \mbox{ in $W_0^{1,p}(\Omega)$ and $y_n \to y$ in $L^p(\Omega)$, $y \geq0$.}
	\end{equation}
	
	From \eqref{eq6} we have
	\begin{align}
		\label{eq10}& \frac{1}{p}	\int_\Omega a_1(z)|\nabla u_n|^p dz+ \frac{1}{q}	\int_\Omega a_2(z)|\nabla u_n|^q dz \leq c_3 +\frac{\lambda}{q}\|u_n^+\|_q^q +\int_\Omega F(z,u_n^+)dz,\\ \nonumber \Rightarrow \quad & \frac{1}{p}	\int_\Omega a_1(z)|\nabla y_n|^p dz+ \frac{1}{q\|u_n^+\|^{p-q}}	\int_\Omega a_2(z)|\nabla y_n|^q dz\\ \label{eq11}
		& \leq \frac{c_3}{\|u_n^+\|^{p}} +\frac{\lambda}{q\|u_n^+\|^{p-q}}\|y_n\|_q^q +\int_\Omega \frac{F(z,u_n^+)}{\|u_n^+\|^{p}}dz \quad \mbox{for all $n \in \mathbb{N}$.}
	\end{align}	
	
	Assumption $H_1(ii)$ leads to
	\begin{align}
		\label{eq12} & \frac{F(\cdot,u_n^+(\cdot))}{\|u_n^+\|^p} \xrightarrow{w} \frac{1}{p} \eta y^p \quad \mbox{in }L^1(\Omega),
		\\ \label{eq13} &	\mbox{with $\eta \in L^\infty(\Omega)$ satisfying $\eta(z)\leq \widehat{\lambda}_1(p,a_1)$ for a.a. $z \in \Omega$.}
	\end{align}
	
	Letting $n \to \infty$ in \eqref{eq11},  by \eqref{eq8}, \eqref{eq9}, \eqref{eq12} and the fact that $q<p$, we deduce that
	\begin{equation}
		\label{eq14}\int_\Omega a_1(z) |\nabla y|^p dz \leq \int_\Omega \eta(z)y^p dz.
	\end{equation}
	
	If $\eta \not \equiv \widehat{\lambda}_1(p,a_1)$ (see \eqref{eq13}), then from \eqref{eq14} one has
	$\widehat{c}\, \|y\|^p \leq 0$  (see Proposition \ref{prop2}), and hence $y=0$.
	From \eqref{eq9} and \eqref{eq11}, we see that $\| \nabla y_n\|_p\to 0$, which leads to contradiction with $\|y_n\|=1$ for all $n \in \mathbb{N}$. 
	
	If $\eta(z)= \widehat{\lambda}_1(p,a_1)$ for a.a. $z \in \Omega$, again from \eqref{eq14} one has
	$ \int_\Omega a_1(z) |\nabla y|^p dz=\widehat{\lambda}_1(p,a_1) \|y\|_p^p,
	$ and hence $y=\vartheta \ \widehat{u}_1(p,a_1)$ for some $\vartheta \geq 0$.

	If $\vartheta=0$, then $y=0$ which leads to contradiction with $\|y_n\|=1$ for all $n \in \mathbb{N}$. 
	
	If $\vartheta>0$, then $y \in {\rm int \,}C_+$ and so we have
	$u_n^+(z)\to+\infty$ for a.a. $z \in \Omega$.
	By $H_1(iii)$ given $\xi>0$, there is $M=M(\xi)>0$ satisfying
	\begin{equation}
		\label{eq16} f(z,x)x-pF(z,x)\geq \xi |x|^\tau \quad \mbox{for a.a. $z \in \Omega$, all $|x| \geq M$.}
	\end{equation}
	
	Additionally
	\begin{align*}
		\frac{d}{dx}\left[\frac{F(z,x)}{|x|^p}   \right] &= \frac{f(z,x)|x|^p-p|x|^{p-2}xF(z,x)}{|x|^{2p}}\\
		&= \frac{f(z,x)x-pF(z,x)}{|x|^px} \, \begin{cases} \geq \frac{\xi}{x^{p-\tau+1}}& \mbox{if }x\geq M\\ \leq \frac{\xi}{|x|^{p-\tau}x}& \mbox{if }x \leq -M\end{cases} \, \mbox{ (see \eqref{eq16}),}\\ \Rightarrow \quad & \frac{F(z,y)}{|y|^p}-\frac{F(z,x)}{|x|^p}\geq -\frac{\xi}{p-\tau}\left[\frac{1}{|y|^{p-\tau}}-\frac{1}{|x|^{p-\tau}} \right]
	\end{align*}
	for a.a. $z \in \Omega$, for all $|y| \geq |x|\geq M$. 
	Letting $|y| \to \infty$, by $H_1(ii)$ we deduce that
	\begin{align}
		\nonumber & \frac{\widehat{\lambda}_1(p,a_1)}{p}-\frac{F(z,x)}{|x|^p}\geq \frac{\xi}{p-\tau} \, \frac{1}{|x|^{p-\tau}},\\ \nonumber \Rightarrow \quad & \frac{\widehat{\lambda}_1(p,a_1)|x|^p-pF(z,x)}{|x|^\tau}\geq \frac{\xi p}{p-\tau} \quad \mbox{for a.a. $z \in \Omega$, all $ |x|\geq M$,}\\ \label{eq17}
		\Rightarrow \quad & \lim_{x \to \pm \infty}\frac{\widehat{\lambda}_1(p,a_1)|x|^p-pF(z,x)}{|x|^\tau}=+\infty \quad \mbox{uniformly for a.a. $z \in \Omega$.}
	\end{align}
	
	Now, \eqref{eq10} gives us
	\begin{align}
		\nonumber & \frac{1}{p} \int_\Omega [\widehat{\lambda}_1(p,a_1)(u_n^+)^p-pF(z,u_n^+) ]dz \leq c_3+\lambda \|u_n^+\|_q^q,\\ \Rightarrow \quad &  	 \frac{1}{p} \int_\Omega \frac{\widehat{\lambda}_1(p,a_1)(u_n^+)^p-pF(z,u_n^+)}{(u_n^+)^\tau}y_n^\tau dz \leq \frac{c_3}{\|u_n^+\|^\tau}+\frac{\lambda c_4}{\|u_n^+\|^{\tau-q}}, \label{eq18}
	\end{align}
	for some $c_4>0$, for all $n \in \mathbb{N}$.
	For $n \to \infty$ in \eqref{eq18}  combining \eqref{eq8}, \eqref{eq17}, Fatou's lemma and recalling that $\tau>q$, we leads to contradiction. The boundedness of $\{u_n^+\}_{n\in \mathbb{N}}\subseteq W_0^{1,p}(\Omega)$ is so established. This implies the boundedness of $\{u_n\}_{n\in \mathbb{N}}\subseteq W_0^{1,p}(\Omega)$  (see \eqref{eq6}), which contradicts \eqref{eq7}. This argument establishes the coercivity of $\varphi_\lambda^+(\cdot)$, as stated in the Claim.
	Next, we observe that $\varphi_\lambda^+(\cdot)$ is sequentially weakly lower semicontinuous (by Sobolev embedding theorem). This fact, the Claim and the Weierstrass-Tonelli theorem, lead to the existence of a $u_\lambda \in W_0^{1,p}(\Omega)$ satisfying
	\begin{equation}
		\label{eq19} \varphi_\lambda^+(u_\lambda)=\inf [\varphi_\lambda^+(u): u \in W_0^{1,p}(\Omega)].
	\end{equation}
	
	So, $H_1(iv)$ for fixed $\varepsilon>0$, gives us $\delta=\delta(\varepsilon)>0$ with
	\begin{equation}
		\label{eq20}|F(z,x)|\leq \frac{\varepsilon}{q}|x|^q \quad \mbox{for a.a. $z \in \Omega$,  all $|x|\leq \delta$.}
	\end{equation}
	
	But $\widehat{u}_1(q,a_2)\in {\rm int \, }C_+$ (Section \ref{sec:2}) ensures there exists $t\in (0,1)$ small enough to get
	\begin{equation}
		\label{eq21} 0 \leq t \widehat{u}_1(q,a_2)(z)\leq \delta \quad \mbox{for all $z \in \overline{\Omega}$.}
	\end{equation}
	
	Therefore,
	$$ \varphi_\lambda^+(t \widehat{u}_1(q,a_2))\leq \frac{t^p}{p} \int_\Omega a_1(z) |\nabla \widehat{u}_1(q,a_2)|^p dz+\frac{t^q}{q} [\widehat{\lambda}_1(q,a_2)+\varepsilon-\lambda]$$ (see \eqref{eq20}, \eqref{eq21}, recall $\|\widehat{u}_1(q,a_2)\|_q=1$).
	If we choose $\varepsilon \in (0, \lambda - \widehat{\lambda}_1(q,a_2))$, then
	\begin{equation}
		\label{eq22}\varphi_\lambda^+(t \widehat{u}_1(q,a_2))\leq c_5 t^p-c_6t^q \quad \mbox{for some $c_5,c_6>0$.}
	\end{equation}
	
	As $p>q$, we choose $t \in (0,1)$ appropriately (i.e., even smaller if necessary), then from \eqref{eq22} we get
	\begin{align*}
		& \varphi_\lambda^+( t \widehat{u}_1(q,a_2))<0,\\ \Rightarrow \quad & \varphi_\lambda^+(u_\lambda)<0=\varphi_\lambda^+(0)\quad \mbox{(recall \eqref{eq19}),}\end{align*}
	and so $u_\lambda \neq 0$.
	Again \eqref{eq19} leads  to $(\varphi_\lambda^+)^\prime(u_\lambda)=0$, which implies
	\begin{equation}
		\label{eq23}   \langle A_{p}^{a_1}(u_\lambda),h \rangle + \langle A_{q}^{a_2}(u_\lambda),h \rangle = \int_\Omega [\lambda (u_\lambda^+)^{q-1}+ f(z,u_\lambda^+)]h dz \quad \mbox{for all }  h \in W_0^{1,p}(\Omega).
	\end{equation}
	
	Equation \eqref{eq23} for the test function $h =-u_\lambda^- \in W_0^{1,p}(\Omega)$, leads to
	the inequality $ c_1 \|\nabla u_\lambda^-\|_p^p \leq 0 $ (see   $H_0$), and hence $u_\lambda \geq 0$, $ u_\lambda \neq 0$.	
	Thus $u_\lambda$ is a positive solution of \eqref{eq0} (see \eqref{eq23}). Ladyzhenskaya and Ural$^\prime$tseva \cite[Theorem 7.1]{Ref10} ensures that $u_\lambda \in L^\infty(\Omega)$. Consequently,  the regularity theory of Lieberman \cite{Ref11} implies $u_\lambda \in C_+\setminus \{0\}$.  Now, Papageorgiou et al. \cite[Proposition 2.2]{Ref18} gives us
	\begin{equation}
		\label{eq24}0< u_\lambda(z) \quad \mbox{for all $z \in \Omega$.}
	\end{equation}
	
	We can continue the proof of \cite[Proposition 2.2]{Ref18}, since now we have more regularity (namely now $u_\lambda \in C_+\setminus \{0\}$). So, let $z_1 \in \partial \Omega$ and set $z_2=z_1-2\rho n$ with $\rho \in (0,1)$ small and $n=n(z_1)$ is the outward unit normal at $z_1$. As in \cite{Ref18}, we consider the annulus $D=\{z \in  \Omega : \rho <|z-z_2|<2\rho\}$ and let $m=\min \{u(z) : z \in \partial B_\rho(z_2)>0  \}$ (see \eqref{eq24}). From the proof in \cite{Ref18}, for $\vartheta \in (0,m)$ small, there is $y \in C^1(\overline{D})\cap C^2(D)$ satisfying the inequality
	$-\Delta_p^{a_1}y-\Delta_q^{a_2}y\leq 0$ in $D$ with $y(z_1)=0$, $\dfrac{\partial y}{\partial n}(z_1)<0$. We know that
	$-\Delta_p^{a_1}u_\lambda-\Delta_q^{a_2}u_\lambda \geq 0$ in $\Omega$.
	So, from the weak comparison principle (Pucci and Serrin \cite{Ref21}, p. 61), one has $y(z)\leq u_\lambda(z)$ for all $z \in D$. It follows that
	$ \dfrac{\partial u_\lambda}{\partial n}(z_1)\leq \dfrac{\partial y}{\partial n}(z_1)<0$, and so $u_\lambda \in {\rm int \,}C_+$.
	Similarly working with  $\varphi_\lambda^- : W_0^{1,p}(\Omega) \to \mathbb{R}$ of the form
	$$ \varphi_\lambda^-(u)= \frac{1}{p}\int_\Omega a_1(z)|\nabla u|^{p}dz + \frac{1}{q}\int_\Omega a_2(z) |\nabla u|^q dz - \frac{\lambda}{q}\|u^-\|_q^q- \int_\Omega F(z,-u^-)dz $$
	for all $u \in W^{1,p}_0(\Omega)$, we get a negative solution $v_\lambda \in -{\rm int \,}C_+$ for problem \eqref{eq0} ($\lambda > \widehat{\lambda}_1(q,a_2)$).
\end{proof}

\begin{remark}
	An alternative way to show that $u_\lambda \in {\rm int \,}C_+$, is the following one. Let 
	$\widehat{d}(z)=d(z,\partial \Omega)$ for all $z \in \overline{\Omega}$. By Gilbarg and Trudinger \cite[Lemma 14.16]{Ref7}, we can find $\delta_0>0$ such that $\widehat{d} \in C^2(\overline{\Omega}_{0})$, where $\Omega_{0}=\{z \in \overline{\Omega}: \widehat{d}(z)\leq \delta_0\}$. It follows that $\widehat{d}\in {\rm int \,} C_+$. From Rademacher's theorem (see Gasi\'{n}ski and Papageorgiou \cite{Ref3}, p. 56), we know that $a_1,a_2$ are both differentiable a.e. in $\Omega$. So, by taking $\delta_0>0$ even smaller if necessary we can have $\dfrac{\partial a_1}{\partial n}\Big|_{\Omega_0},\dfrac{\partial a_2}{\partial n}\Big|_{\Omega_0}\leq 0$. On account of \eqref{eq24},   we can find $t \in (0,1)$ small such that $w= t \, \widehat{d}\leq \overline{u}_\lambda$ on $\partial\Omega_0$. Additionally,    \cite[Lemma 14.17]{Ref7} leads to
	$$ 	-\Delta_p^{a_1}w-\Delta_q^{a_2}w \leq 0 \leq -\Delta_p^{a_1}u_\lambda-\Delta_q^{a_2}u_\lambda\quad \mbox{in $\Omega_0$ $($see $H_1(v))$,}\quad w \leq u_\lambda \quad \mbox{on $\partial \Omega_0$.}
	$$
	
	Then the weak comparison principle (see Pucci and Serrin \cite{Ref21}, p. 61), gives us $w \leq u_\lambda$ in $\Omega_0$. Hence for a certain $\widehat{t}\in (0,1)$ small satisfying $\widehat{t} \, \widehat{d}\leq u_\lambda$ in $\Omega$, we get $u_\lambda \in {\rm int \,}C_+$.
\end{remark}

We now establish the existence of   smallest positive and  biggest negative solutions.  From
$H_1(iv)$ and \eqref{eq5}, fixed $\varepsilon>0$, there exists a constant $c_7=c_7(\varepsilon)>0$ satisfying
\begin{align}
	& \nonumber f(z,x)x \geq -\varepsilon |x|^q-c_7|x|^p \quad \mbox{for a.a. $z \in \Omega$, all $x \in \mathbb{R}$,}\\ \label{eq25} \Rightarrow \quad & \lambda |x|^q+ f(z,x)x \geq [\lambda-\varepsilon]|x|^q-c_7|x|^p \quad \mbox{for a.a. $z \in \Omega$, all $x \in \mathbb{R}$.}
\end{align}

Observe that \eqref{eq25} leads to the following auxiliary Dirichlet problem
\begin{equation}\label{eq26} 
	\begin{cases}-\Delta_{p}^{a_1} u(z)-\Delta_{q}^{a_2} u(z)=[\lambda-\varepsilon] |u(z)|^{q-2} u(z)-c_7 |u(z)|^{p-2} u(z) & \mbox{in } \Omega,\\ u \Big|_{\partial \Omega} =0,  \,  \lambda>0.& \end{cases}
\end{equation}

\begin{proposition}
	\label{prop4}Let $H_0$ be satisfied, $\lambda >\widehat{\lambda}_1(q,a_2)$ and $\varepsilon \in (0,\lambda- \widehat{\lambda}_1(q,a_2))$. Then \eqref{eq26} admits a unique positive solution $\overline{u}_\lambda \in  {\rm int \,}C_+$. Additionally, as \eqref{eq26} is odd, then it admits a unique negative solution $\overline{v}_\lambda=-\overline{u}_\lambda \in - {\rm int \,}C_+$.
\end{proposition}

\begin{proof}
	We start discussing  the existence of a positive solution for problem \eqref{eq26}. To this end let $\psi_\lambda^+ : W_0^{1,p}(\Omega) \to \mathbb{R}$ defined by 
	$$ \psi_\lambda^+(u)= \frac{1}{p}\int_\Omega a_1(z)|\nabla u|^{p}dz + \frac{1}{q}\int_\Omega a_2(z) |\nabla u|^q dz - \frac{\lambda-\varepsilon}{q}\|u^+\|_q^q + \frac{c_7}{p}\|u^+\|_p^p $$
	for all $u \in W^{1,p}_0(\Omega)$. Since $q<p$, we see that $\psi_\lambda^+(\cdot)$ is coercive. Also, it is sequentially weakly lower semicontinuous. By using the similar arguments to the ones in the proof of Proposition \ref{prop3}, one can find
	$\overline{u}_\lambda \in W^{1,p}_0(\Omega)$ positive solution to \eqref{eq26} (i.e., $\overline{u}_\lambda\geq 0$, $\overline{u}_\lambda\neq 0$) and also $\overline{u}_\lambda \in {\rm int \,}C_+$.  To establish the uniqueness of $\overline{u}_\lambda$, we need the functional $j: L^1(\Omega)\to \overline{\mathbb{R}}=\mathbb{R} \cup \{+\infty\}$ of the form
	\begin{equation}\label{func-j}j(u)= \begin{cases}  \frac{1}{p}\int_\Omega a_1(z)|\nabla u^{1/q}|^{p}dz+\frac{1}{q}\int_\Omega a_2(z)|\nabla u^{1/q}|^q dz & \mbox{if }u \geq 0, \, u^{1/q} \in W_0^{1,p}(\Omega),\\ +\infty & \mbox{otherwise.} 
		\end{cases}
	\end{equation}
	The convexity of \eqref{func-j} follows from D\'{i}az and  Sa\'{a} \cite[Lemma 1]{Ref2}. We introduce ${\rm dom \,}j=\{u \in L^1(\Omega) : j(u)<+ \infty \}$ and argue by contradiction. Suppose that $\overline{w}_\lambda$ is another positive solution of \eqref{eq26}.  Of course,   $\overline{w}_\lambda \in  {\rm int \, }C_+$ and Papageorgiou et al. \cite[Proposition 4.1.22]{Ref16} give us
	$ \dfrac{\overline{u}_\lambda}{\overline{w}_\lambda}\in L^\infty(\Omega)$ and  $\dfrac{\overline{w}_\lambda}{\overline{u}_\lambda}\in L^\infty(\Omega)$.
	Hence if $h= \overline{u}_\lambda^q-\overline{w}_\lambda^q$, a sufficiently small  $|t|<1$ leads to $ \overline{u}_\lambda^q + th \in {\rm dom \,}j,\, \overline{w}_\lambda^q + th \in {\rm dom \,}j$. Since \eqref{func-j} is convex, we have that it is also Gateaux differentiable (in the direction $h$) at  $\overline{u}_\lambda^q$ and at $\overline{w}_\lambda^q$. Using  chain rule together with nonlinear Green's identity (\cite{Ref16}, p. 35), one has
	\begin{align*}
		j^\prime(\overline{u}_\lambda)(h) & = \frac{1}{q}\int_\Omega \frac{-\Delta^{a_1}_p \overline{u}_\lambda -\Delta^{a_2}_q \overline{u}_\lambda }{\overline{u}_\lambda^{q-1}}h dz =  \int_\Omega ([\lambda-\varepsilon ] -c_7 \overline{u}_\lambda^{p-q})h dz,\\
		j^\prime(\overline{w}_\lambda)(h) & = \frac{1}{q}\int_\Omega \frac{-\Delta^{a_1}_p \overline{w}_\lambda -\Delta^{a_2}_q \overline{w}_\lambda }{\overline{w}_\lambda^{q-1}}h dz =  \int_\Omega ([\lambda-\varepsilon ] -c_7 \overline{w}_\lambda^{p-q})h dz.
	\end{align*}
	Since \eqref{func-j} is convex, then $j^\prime(\cdot)$ is monotone, and so
	$$ 0 \leq \int_\Omega c_7[\overline{w}_\lambda^{p-q}-\overline{u}_\lambda^{p-q}] (\overline{u}_\lambda^q -\overline{w}_\lambda^q) dz \leq 0,$$ which implies that $\overline{u}_\lambda= \overline{w}_\lambda$.
	We conclude that \eqref{eq26} admits a unique positive solution $\overline{u}_\lambda \in  {\rm int \, }C_+$. By oddness of \eqref{eq26}, we deduce that it admits a unique negative solution $\overline{v}_\lambda= -\overline{u}_\lambda \in - {\rm int \, }C_+$. 
\end{proof}

In the sequel, we will work with:
\begin{align*}
	& \mathcal{S}_\lambda^+=\{\mbox{set of positive solutions to \eqref{eq0}}\},\\
	& \mathcal{S}_\lambda^-=\{\mbox{set of negative solutions to \eqref{eq0}}\}.
\end{align*}

Observe (by Proposition \ref{prop3}) that if $\lambda > \widehat{\lambda}_1(q,a_2)$, then
$\emptyset \neq \mathcal{S}_\lambda^+ \subseteq {\rm int \, }C_+$ and $ \emptyset \neq \mathcal{S}_\lambda^- \subseteq -{\rm int \, }C_+.$
We also mention that the unique constant sign solutions of \eqref{eq26} provide bounds for the elements of these two solution sets.

\begin{proposition}
	\label{prop5} Let $H_0$, $H_1$ be satisfied, and $\lambda > \widehat{\lambda}_1(q,a_2)$. Then $\overline{u}_\lambda \leq u$ for all $u \in \mathcal{S}_\lambda^+$ and $v \leq \overline{v}_\lambda$  for all $v \in \mathcal{S}_\lambda^-$.
\end{proposition}

\begin{proof}
	
	For $u \in \mathcal{S}_\lambda^+ \subseteq {\rm int \, }C_+$ and $\varepsilon \in (0,\lambda- \widehat{\lambda}_1(q,a_2))$, we introduce a Carath\'eodory function $k_\lambda^+:\Omega \times \mathbb{R}\to \mathbb{R}$ defined by
	\begin{equation}
		\label{eq28}k_\lambda^+(z,x)=\begin{cases}[\lambda-\varepsilon](x^+)^{q-1}-c_7 (x^+)^{p-1}& \mbox{if }x \leq u(z),\\ [\lambda-\varepsilon]u(z)^{q-1}-c_7 u(z)^{p-1} & \mbox{if }u(z)<x.\end{cases}
	\end{equation}
	
	Let $K_\lambda^+(z,x)=\int_0^x k_\lambda^+(z,s)ds$ and $\beta_\lambda^+:W_0^{1,p}(\Omega)\to \mathbb{R}$ be the $C^1$-functional 
	$$ \beta_\lambda^+(u)=\frac{1}{p} \int_\Omega a_1(z)|\nabla u|^pdz+\frac{1}{q}\int_\Omega a_2(z)|\nabla u|^q dz-\int_\Omega K_\lambda^+(z,u)dz $$
	for all $u \in W_0^{1,p}(\Omega)$.
	Now \eqref{eq28} ensures the coercivity of $\beta_\lambda^+(\cdot)$; additionally, $\beta_\lambda^+(\cdot)$ is sequentially weakly lower semicontinuous.
	By using the similar arguments to the ones in the proof of Proposition \ref{prop3}, one can deduce that there exists $\widetilde{u}_\lambda \in W_0^{1,p}(\Omega)$ with
	\begin{equation} \langle A_p^{a_1}(\widetilde{u}_\lambda),h \rangle + \langle A_q^{a_2}(\widetilde{u}_\lambda),h \rangle = \int_\Omega k_\lambda^+(z,\widetilde{u}_\lambda)h dz \quad \mbox{for all } h \in W_0^{1,p}(\Omega).\label{eq30} \end{equation}
	
	In \eqref{eq30} first we use $h=-\widetilde{u}_\lambda^- \in W_0^{1,p}(\Omega)$ leading to  $\widetilde{u}_\lambda\geq 0$, $\widetilde{u}_\lambda \neq 0$.	Next taking $h=(\widetilde{u}_\lambda-u)^+ \in W_0^{1,p}(\Omega)$, we have
	\begin{align*}
		& \langle A_p^{a_1}(\widetilde{u}_\lambda),(\widetilde{u}_\lambda-u)^+ \rangle  +\langle A_q^{a_2}(\widetilde{u}_\lambda),(\widetilde{u}_\lambda-u)^+ \rangle  \\ & = \int_\Omega \left([\lambda-\varepsilon]u^{q-1}-c_7u^{p-1}\right)(\widetilde{u}_\lambda-u)^+ dz \quad \mbox{(see \eqref{eq28})}\\ & \leq \int_\Omega \left(\lambda u^{q-1}+f(z,u)\right)(\widetilde{u}_\lambda-u)^+ dz \quad \mbox{(see \eqref{eq25})}\\ & = \langle A_p^{a_1}(u),(\widetilde{u}_\lambda-u)^+ \rangle  +\langle A_q^{a_2}(u),(\widetilde{u}_\lambda-u)^+ \rangle \quad \mbox{(since $u \in \mathcal{S}_\lambda^+$),} 	\end{align*}
	which implies $\widetilde{u}_\lambda \leq u$.
	Summarizing
	\begin{equation}
		\label{eq31}\widetilde{u}_\lambda \in [0,u], \quad \widetilde{u}_\lambda \neq 0.
	\end{equation}
	
	Using \eqref{eq28}, \eqref{eq31}, \eqref{eq30}, then $\widetilde{u}_\lambda$ is  positive solution of \eqref{eq26}. So, on account of Proposition \ref{prop4}, we have $\widetilde{u}_\lambda=\overline{u}_\lambda$. Therefore  
	$\overline{u}_\lambda \leq u$  for all $u \in \mathcal{S}_\lambda^+$ (see \eqref{eq31}). Clearly, on the similar lines, one can establish that
	$v \leq \overline{v}_\lambda$ for all $v \in \mathcal{S}_\lambda^-$.
\end{proof}	

The extremal constant sign solutions to \eqref{eq0} ($\lambda > \widehat{\lambda}_1(q,a_2)$) are obtained as follows.
\begin{proposition}
	\label{prop6} Let $H_0$, $H_1$ be satisfied, and $\lambda > \widehat{\lambda}_1(q,a_2)$. Then there exist $u_\lambda^\ast \in \mathcal{S}_\lambda^+$ and $v_\lambda^\ast \in \mathcal{S}_\lambda^-$ where $u_\lambda^\ast \leq u$ for all $u \in \mathcal{S}_\lambda^+$, $v \leq v_\lambda^\ast$ for all $v \in \mathcal{S}_\lambda^-$.
\end{proposition}

\begin{proof}
	We mention that Papageorgiou et al. \cite[Proposition 7]{Ref15} ensures that $\mathcal{S}_\lambda^+$ is downward directed (i.e., if $u_1,u_2 \in \mathcal{S}_\lambda^+$, then there exists $u \in \mathcal{S}_\lambda^+$ with $u \leq u_1$, $u \leq u_2$). Moreover, Hu and Papageorgiou \cite[Lemma 3.10]{Ref8} help us to find $\{u_n\}_{n \in \mathbb{N}}\subseteq \mathcal{S}_\lambda^+ \subseteq {\rm int \,}C_+$ decreasing and satisfying
	\begin{equation}
		\label{eq32}  \inf\limits_{n \in \mathbb{N}}u_n =\inf \mathcal{S}_\lambda^+, \quad \overline{u}_\lambda \leq u_n \leq u_1 \quad \mbox{for all  $n \in \mathbb{N}$ (see Proposition \ref{prop5}).}
	\end{equation}
	
	Starting from
	\begin{equation}
		\label{eq33}\langle A_{p}^{a_1}(u_n),h \rangle + \langle A_{q}^{a_2}(u_n),h \rangle = \int_\Omega [\lambda u_n^{q-1}+ f(z,u_n)] hdz \quad \mbox{for all }  h \in W_0^{1,p}(\Omega),
	\end{equation}
	and taking $h =u_n\in W_0^{1,p}(\Omega)$, then \eqref{eq32} and  $H_0$ give us 
	$c_1 \| \nabla u_n\|_p^p \leq c_8$ for some $c_8>0$, for all $n \in \mathbb{N}$,
	and hence $\{u_n\}_{n \in \mathbb{N}}\subseteq W_0^{1,p}(\Omega)$ is bounded.
	Therefore, it is possible to suppose 
	\begin{equation}
		\label{eq34} u_n \xrightarrow{w} u_\lambda^\ast \mbox{ in $W_0^{1,p}(\Omega)$, $u_n \to u_\lambda^\ast$ in $L^p(\Omega)$.}
	\end{equation}
	
	Before taking $n \to \infty$ in \eqref{eq33}, we use $h =u_n-u_\lambda^\ast \in W_0^{1,p}(\Omega)$, and by \eqref{eq34} we get
	\begin{align}
		\nonumber & \lim_{n\to \infty} [\langle A_{p}^{a_1}(u_n),u_n-u_\lambda^\ast \rangle + \langle A_{q}^{a_2}(u_n),u_n-u_\lambda^\ast \rangle	]=0, \\ \nonumber \Rightarrow \quad & \limsup_{n\to \infty} [\langle A_{p}^{a_1}(u_n),u_n-u_\lambda^\ast \rangle + \langle A_{q}^{a_2}(u_\lambda^\ast),u_n-u_\lambda^\ast \rangle	]\leq 0 \quad \mbox{(since $A_{q}^{a_2}(\cdot)$ is monotone),}\\ \nonumber \Rightarrow \quad & \limsup_{n\to \infty}  \langle A_{p}^{a_1}(u_n),u_n-u_\lambda^\ast \rangle  \leq 0 \quad \mbox{(see \eqref{eq34}),}\\ \label{eq35} \Rightarrow \quad & u_n \to u_\lambda^\ast \mbox{ in $W_0^{1,p}(\Omega)$ ($A_{p}^{a_1}$ is of type $(S)_+$).}
	\end{align}
	
	Returning to Eq \eqref{eq33} and letting again $n \to \infty$, \eqref{eq35} and \eqref{eq32} lead to
	\begin{align*}
		& \langle A_{p}^{a_1}(u_\lambda^\ast),h \rangle + \langle A_{q}^{a_2}(u_\lambda^\ast),h \rangle = \int_\Omega [\lambda (u_\lambda^\ast)^{q-1}+ f(z,u_\lambda^\ast)]h dz \quad \mbox{for all }  h \in W_0^{1,p}(\Omega), \\ 
		& \overline{u}_\lambda \leq u_\lambda^\ast.
	\end{align*}
	
	We arrive to the conclusion that $ u_\lambda^\ast \in \mathcal{S}_\lambda^+$ and $ u_\lambda^\ast =\inf \mathcal{S}_\lambda^+$.	Similarly, we produce $v_\lambda^\ast \in \mathcal{S}_\lambda^-$, $v_\lambda^\ast =\sup \mathcal{S}_\lambda^-$, where $\mathcal{S}_\lambda^-$ is upward directed (i.e., if $v_1,v_2 \in \mathcal{S}_\lambda^-$, then there exists $v \in \mathcal{S}_\lambda^-$ with $v_1\leq v$, $v_2\leq v$).
\end{proof} 

\section{Nodal solutions}\label{sec:4}

We implement a simple idea: we will use truncations to work over the order interval $[v_\lambda^\ast,u_\lambda^\ast]$. Any nontrivial solution ($\not \equiv \,u_\lambda^\ast$, $v_\lambda^\ast$) of \eqref{eq0} there, will be nodal. The key ingredient is the minimax characterization of $\widehat{\lambda}_2(q,a_2)$ (see \eqref{eq4}).
From Section \ref{sec:3} we have $u_\lambda^\ast \in {\rm int \, }C_+$ and $v_\lambda^\ast \in - {\rm int \, }C_+$ solving \eqref{eq0} ($\lambda > \widehat{\lambda}_1(q,a_2)$). Then we introduce
\begin{equation}
	\label{eq36}\mu_\lambda(z,x)=\begin{cases}
		\lambda |v_\lambda^\ast(z)|^{q-2}v_\lambda^\ast(z)+f(z, v^\ast_\lambda(z)) & \mbox{if } x<v^\ast_\lambda(z),\\
		\lambda |x|^{q-2}x+f(z, x)  & \mbox{if } v^\ast_\lambda(z)\leq x\leq u^\ast_\lambda(z),\\
		\lambda u_\lambda^\ast(z)^{q-1}+f(z, u^\ast_\lambda(z)) & \mbox{if } u^\ast_\lambda(z) <x.
	\end{cases}
\end{equation}

Evidently $\mu_\lambda(\cdot,\cdot)$ is of Carath\'eodory. Additionally, we need 
\begin{equation}
	\label{eq37}\mu_\lambda^\pm(z,x)=\mu_\lambda(z,\pm x^\pm).
\end{equation}
Putting $M_\lambda(z,x)=\int_0^x\mu_\lambda(z,s)ds$, $M_\lambda^\pm(z,x)=\int_0^x\mu_\lambda^\pm(z,s)ds$, one can define the $C^1$-functionals $\widehat{\psi}_\lambda,\widehat{\psi}_\lambda^\pm :W^{1,p}_0(\Omega)\to \mathbb{R}$ as
\begin{align*}
	&	\widehat{\psi}_\lambda(u)=\frac{1}{p}\int_\Omega a_1(z)|\nabla u|^p dz+  \frac{1}{q}\int_\Omega a_2(z)|\nabla u|^q dz -\int_\Omega M_\lambda(z,u)dz,\\
	&\widehat{\psi}_\lambda^\pm(u)=\frac{1}{p}\int_\Omega a_1(z)|\nabla u|^p dz+  \frac{1}{q}\int_\Omega a_2(z)|\nabla u|^q dz   -\int_\Omega M_\lambda^\pm(z,u)dz 
\end{align*}
for all $u \in W^{1,p}_0(\Omega)$.
From  \eqref{eq36}, \eqref{eq37}, the nonlinear regularity theory and the extremality of $u^\ast_\lambda$ and $v^\ast_\lambda$, we infer easily the following result.

\begin{proposition}
	\label{prop7} Let $H_0$, $H_1$ be satisfied, and $\lambda > \widehat{\lambda}_1(q,a_2)$. Then,  $K_{\widehat{\psi}_\lambda}\subseteq [v_\lambda^\ast,u_\lambda^\ast]\cap C_0^1(\overline{\Omega})$, $K_{\widehat{\psi}^+_\lambda}=\{0,u^\ast_\lambda\}$, 
	$K_{\widehat{\psi}^-_\lambda}=\{0,v^\ast_\lambda\}$.\end{proposition}

We establish the following auxiliary proposition.

\begin{proposition}
	\label{prop8} Let $H_0$, $H_1$ be satisfied, and $\lambda > \widehat{\lambda}_1(q,a_2)$. Then,  $u_\lambda^\ast \in {\rm int \,}C_+$ and  $v_\lambda^\ast \in - {\rm int \,}C_+$
	are local minimizers of  $\widehat{\psi}_\lambda(\cdot)$.\end{proposition}

\begin{proof}
	Definitions \eqref{eq36} and \eqref{eq37} give us the coercivity of $\widehat{\psi}^\pm_\lambda(\cdot)$, which are sequentially weakly lower semicontinuous too. Similarly to the proofs of previous propositions but 	involving $\widehat{\psi}_\lambda^+(\cdot)$ this time, there exists  a certain $\widetilde{u}^\ast_\lambda \in W^{1,p}_0(\Omega)$ with 
	$\widetilde{u}^\ast_\lambda \neq 0$.
	As $\widetilde{u}^\ast_\lambda \in K_{\widehat{\psi}_\lambda^+}\setminus \{0\}$, from Proposition \ref{prop7}, we get $\widetilde{u}^\ast_\lambda=u^\ast_\lambda \in {\rm int \,}C_+$.
	Observe $\widehat{\psi}_\lambda \Big|_{C_+}= \widehat{\psi}_\lambda^+ \Big|_{C_+}$ (see \eqref{eq36}, \eqref{eq37}), and hence we have
	\begin{align*}
		& u^\ast_\lambda \mbox{ is a local $C_0^1(\overline{\Omega})$-minimizer of $\widehat{\psi}_\lambda(\cdot)$,}\\  \Rightarrow  \quad  & u^\ast_\lambda \mbox{ is a local $W_0^{1,p}( \Omega)$-minimizer of $\widehat{\psi}_\lambda(\cdot)$ (refer to \cite{Ref4}).} 
	\end{align*}
	Involving in a similar way $\widehat{\psi}_\lambda^-(\cdot)$, we complete the proof for  $v^\ast_\lambda \in - {\rm int \,}C_+$.
\end{proof} 

Using the method outlined in the beginning of this section, we establish the following.

\begin{proposition}
	\label{prop9} Let $H_0$, $H_1$ be satisfied, and $\lambda > \widehat{\lambda}_2(q,a_2)$. Then,  \eqref{eq0} admits a nodal solution $y_\lambda \in [v_\lambda^\ast,u_\lambda^\ast]\cap C_0^1(\overline{\Omega})$.\end{proposition}

\begin{proof} To develop the reasoning here, we start from the inequality
	\begin{equation}
		\label{eq38}\widehat{\psi}_\lambda(v_\lambda^\ast)\leq \widehat{\psi}_\lambda(u_\lambda^\ast),
	\end{equation}	
	but of course we could assume equivalently $\widehat{\psi}_\lambda(v_\lambda^\ast)\geq \widehat{\psi}_\lambda(u_\lambda^\ast)$. On account of Proposition \ref{prop7} and without any restriction, let
	$K_{\widehat{\psi}_\lambda}$ be finite (otherwise we already have an infinity of nodal smooth solutions).  Proposition \ref{prop8}, \eqref{eq38} and Papageorgiou et al. \cite[Theorem 5.7.6]{Ref16}, ensure us that there is $\rho \in (0,1)$ small with
	\begin{equation}
		\label{eq40}\widehat{\psi}_\lambda(v_\lambda^\ast)\leq \widehat{\psi}_\lambda(u_\lambda^\ast)< \inf [\widehat{\psi}_\lambda(u) : \|u-u_\lambda^\ast\|=\rho]=\widehat{m}_\lambda, \quad \rho <\|v_\lambda^\ast-u_\lambda^\ast\| \quad\mbox{(see \eqref{eq38}).}
	\end{equation}
	
	Again definition \eqref{eq36} gives us the coercivity of  $\widehat{\psi}_\lambda(\cdot)$, which hence satisfies the Palais-Smale condition (\cite{Ref16}, p. 369). This fact and \eqref{eq40} lead to a mountain pass geometry, which ensures the existence of $y_\lambda \in W_0^{1,p}(\Omega)$ with
	\begin{equation}
		\label{eq41}y_\lambda \in K_{\widehat{\psi}_\lambda}\subseteq [v_\lambda^\ast,u_\lambda^\ast]\cap C_0^1(\overline{\Omega})\quad \mbox{(see Proposition \ref{prop7}),}\quad \widehat{m}_\lambda \leq \widehat{\psi}_\lambda(y_\lambda).
	\end{equation}
	
	From \eqref{eq41} and \eqref{eq36} it follows that $y_\lambda \in C_0^1(\overline{\Omega})$ solves \eqref{eq0} and it is distinct from $u_\lambda^\ast$, $v_\lambda^\ast$. To conclude, it remains to prove that $y_\lambda \neq 0$. Mountain pass theorem ensures that
	$$\widehat{\psi}_\lambda(y_n)=\inf\limits_{\gamma \in \Gamma}\max\limits_{-1\leq t \leq 1}\widehat{\psi}_\lambda(\gamma(t)),$$
	with $\Gamma=\{\gamma \in C([-1,1],W_0^{1,p}(\Omega)) : \gamma(-1)=v_\lambda^\ast, \, \gamma(1)=u_\lambda^\ast \}$. We consider the following Banach manifolds
	$M=W_0^{1,p}(\Omega) \cap \partial B_1^{L^q}$, $M_c=M \cap C_0^1(\overline{\Omega})$,
	where $\partial B_1^{L^q}=\{u \in L^q(\Omega) : \|u\|_q=1 \}$ and we introduce the sets of paths:
	\begin{align*}
		& \widehat{\Gamma}=\{\widehat{\gamma}\in C([-1,1],M) : \widehat{\gamma}(-1)=-\widehat{u}_1(q,a_2), \, \widehat{\gamma}(1)=\widehat{u}_1(q,a_2) \},\\
		& \widehat{\Gamma}_c=\{\widehat{\gamma}\in C([-1,1],M_c) : \widehat{\gamma}(-1)=-\widehat{u}_1(q,a_2), \, \widehat{\gamma}(1)=\widehat{u}_1(q,a_2) \}.
	\end{align*}
	
	\smallskip
	
	\noindent \underline{Claim:} $\widehat{\Gamma}_c$ is dense in $\widehat{\Gamma}$.
	
	Given $\widehat{\gamma}\in \widehat{\Gamma}$ and $\varepsilon \in (0,1)$, we introduce $\widehat{K}_\varepsilon:[-1,1]\to 2^{C_0^1(\overline{\Omega})}$ of the form
	$$\widehat{K}_\varepsilon(t)=\begin{cases} \{u \in C^1(\overline{\Omega}) : \| u - \widehat{\gamma}(t)\|<\varepsilon  \} & \mbox{if } -1<t<1,\\ \{ \pm \widehat{u}_1(q,a_2)\} & \mbox{if }t=\pm 1.
	\end{cases}$$
	
	This multifunction has nonempty and convex values. Additionally, for $t \in (-1,1)$ $\widehat{K}_\varepsilon(t)$ is open, while the sets $\widehat{K}_\varepsilon(1)$, $\widehat{K}_\varepsilon(-1)$ are singletons. Now, Hu and  Papageorgiou \cite[Proposition 2.6]{Ref8}, implies that $\widehat{K}_\varepsilon(\cdot)$ is lsc, and hence  Michael \cite[Theorem 3.1$^{\prime\prime\prime}$]{Ref13} ensures the existence of a continuous map $\widehat{\gamma}_\varepsilon: [-1,1]\to C_0^1(\overline{\Omega})$ with
	$\widehat{\gamma}_\varepsilon(t) \in \widehat{K}_\varepsilon(t)$  for all $t \in [-1,1]$.
	
	Put $\varepsilon=n^{-1}$, $n \in \mathbb{N}$ and let $\widehat{\gamma}_n=\widehat{\gamma}_{\frac{1}{n}}$ be the continuous selection of the multifunction $\widehat{K}_{\frac{1}{n}}(\cdot)$ produced above. The inequality
	\begin{equation}
		\label{eq42}\|\widehat{\gamma}_n(t)-\widehat{\gamma}(t)\|<\frac{1}{n}\quad \mbox{for all $t\in [-1,1]$,}
	\end{equation}
	holds and since $\widehat{\gamma}\in \widehat{\Gamma}$, we see that $\|\widehat{\gamma}(t)\|\geq m>0$ for all $t \in [-1,1]$. Hence \eqref{eq42} leads us to suppose $\|\widehat{\gamma}_n(t)\|\neq 0$ for all $t \in [-1,1]$, all $n \in \mathbb{N}$. We set
	$\widetilde{\gamma}_n(t)=\dfrac{\widehat{\gamma}_n(t)}{\|\widehat{\gamma}_n(t)\|_q}$ for all $t \in [-1,1]$, all $n \in \mathbb{N}$. Then we have
	$\widetilde{\gamma}_n \in C([-1,1],M_c)$, $\widetilde{\gamma}_n(\pm1)=\pm \widehat{u}_1(q,a_2).$ Moreover,
	\begin{align}
		\nonumber \|\widetilde{\gamma}_n(t)-\widehat{\gamma}(t)\| &\leq \|\widetilde{\gamma}_n(t)-\widehat{\gamma}_n(t)\|+\|\widehat{\gamma}_n(t)-\widehat{\gamma}(t)\|\\
		\label{eq43} & \leq \frac{|1-\|\widehat{\gamma}_n(t)\|_q|}{\|\widehat{\gamma}_n(t)\|_q}\|\widehat{\gamma}_n(t)\|+\frac{1}{n} \quad \mbox{for all $t \in [-1,1]$, all $n \in \mathbb{N}$ (see \eqref{eq42}).}
	\end{align}
	
	Note that
	\begin{align*}
		\max\limits_{-1 \leq t \leq 1} |1-\|\widehat{\gamma}_n(t)\|_q|
		& = \max\limits_{-1 \leq t \leq 1} |\|\widehat{\gamma}(t)\|_q-\|\widehat{\gamma}_n(t)\|_q| \quad \mbox{(since $\widehat{\gamma}\in \widehat{\Gamma}$)}\\
		& \leq  \max\limits_{-1 \leq t \leq 1} \|\widehat{\gamma}(t)-\widehat{\gamma}_n(t)\|_q\\
		& \leq c_9 \max\limits_{-1 \leq t \leq 1} \|\widehat{\gamma}(t)-\widehat{\gamma}_n(t)\| \quad \mbox{for some $c_9>0$ ($W_0^{1,q}(\Omega)\hookrightarrow L^q(\Omega)$)}\\ &\leq \frac{c_9}{n}\quad \mbox{(see \eqref{eq42}).}
	\end{align*}
	
	We use this estimate in \eqref{eq43}, together with \eqref{eq42} and the fact that $W_0^{1,p}(\Omega)\hookrightarrow L^q(\Omega)$. We obtain
	$$ \|\widetilde{\gamma}_n(t)-\widehat{\gamma}(t)\| \leq \frac{c_9}{nc_{10}-1}\left[1+\frac{1}{n}\right]+\frac{1}{n}\quad\mbox{for some $c_{10}>0$, all $n \in \mathbb{N}$,}
	$$ which implies that $\widehat{\Gamma}_c$ is dense in $\widehat{\Gamma}$.
	Using this and \eqref{eq4}, one can find $\widehat{\gamma}\in \widehat{\Gamma}_c$ satisfying
	\begin{align*}
		& \label{eq44}\int_\Omega a_2(z)|\nabla \widehat{\gamma}(t)|^qdz<\widehat{\lambda}_2(q,a_2)+\vartheta \quad \mbox{for all $t \in [-1,1]$, with $0<\vartheta<\dfrac{1}{2}(\lambda -\widehat{\lambda}_2(q,a_2))$.}
	\end{align*}
	
	Next, $H_1(iv)$ ensures the existence of $\delta>0$ satisfying
	\begin{equation}
		\label{eq45}F(z,x)\geq -\frac{\vartheta}{q}|x|^q\quad \mbox{for a.a. $z \in \Omega$, all $|x|\leq \delta$.}
	\end{equation}
	
	We have the compactness of $\widehat{\gamma}([-1,1])\subseteq M_c$, and we know that $u_\lambda^\ast \in {\rm int \,}C_+$ and $v_\lambda^\ast \in -{\rm int \,}C_+$. Now, by Papageorgiou et al. \cite[Proposition 4.1.24]{Ref16}, we can find $\xi \in (0,1)$ small with
	\begin{equation}\begin{gathered}
			\label{eq46} \xi \,\widehat{\gamma}(t)\in [v_\lambda^\ast,u_\lambda^\ast]\cap C_0^1(\overline{\Omega})\quad \mbox{for all $t \in [-1,1]$,}\\ |\xi \, \widehat{\gamma}(t)(z)|\leq \delta \quad \mbox{for all $t \in [-1,1]$, all $z \in \overline{\Omega}$.}
		\end{gathered}
	\end{equation}	
	
	Consider $u \in \xi \,\widehat{\gamma}([-1,1])$. Therefore $u=\xi \, \widehat{u}$ with $\widehat{u}\in \widehat{\gamma}([-1,1])$. We have
	\begin{align*}
		\widehat{\psi}_\lambda(u) & \leq \frac{\xi^p}{p}\int_\Omega a_1(z)|\nabla \widehat{u}|^p dz +\frac{\xi^q}{q}\left[\int_\Omega a_2(z)|\nabla \widehat{u}|^q dz -(\lambda-\vartheta)\right]\\ & \hskip 4cm \mbox{(see \eqref{eq45}, \eqref{eq46} and recall $\|\widehat{\gamma}(t)\|_q=1$)}\\
		& \leq \frac{\xi^p}{p}\int_\Omega a_1(z)|\nabla \widehat{u}|^p dz -\frac{\xi^q}{q}\left[\lambda-(\widehat{\lambda}_2(q,a_2)+2\vartheta)\right]\quad \mbox{(see again \eqref{eq45}, \eqref{eq46})}\\ & \leq c_{11}\xi^p-c_{12}\xi^q \quad \mbox{for some $c_{11}, c_{12}>0$ (recall the choice of $\vartheta$).}
	\end{align*}
	
	Then choosing $\xi \in (0,1)$ (smaller enough), one has
	\begin{equation}
		\label{eq47} \widehat{\psi}_\lambda \Big|_{\gamma_0}<0 \quad \mbox{where $\gamma_0=\xi \,\widehat{\gamma}$.}
	\end{equation}
	
	Let $a=\widehat{\psi}_\lambda^+(u_\lambda^\ast)=\widehat{\psi}_\lambda(u_\lambda^\ast)$ and $b=0=\widehat{\psi}_\lambda^+(0)=\widehat{\psi}_\lambda(0)$. From the proof of Proposition \ref{prop8}, we know that $a<b=0$. Moreover on account of Proposition \ref{prop7} and since $u_\lambda^\ast$ is the global minimizer of $\widehat{\psi}_\lambda^+$, one can conclude that
	$K^a_{\widehat{\psi}_\lambda^+}=\{u_\lambda^\ast\}$, $ \widehat{\psi}_\lambda^+(K_{\widehat{\psi}_\lambda^+})\cap (a,0)=\emptyset.$
	
	Therefore we can apply the second deformation theorem in Papageorgiou et al. \cite{Ref16} (p. 386) and produce  $h_0:[0,1 ]\times ((\widehat{\psi}^+_\lambda)^0\setminus K^0_{\widehat{\psi}_\lambda^+})\to (\widehat{\psi}_\lambda^+)^a$ such that
	\begin{align}
		& \label{eq48}h_0(0,u)=u \quad \mbox{for all $u \in ((\widehat{\psi}_\lambda^+)^0\setminus\{0\})$ (note $K^0_{\widehat{\psi}_\lambda^+}=\{0\}$),}\\ & \label{eq49}h_0(t,u)=u_\lambda^\ast \quad \mbox{for all $u \in ((\widehat{\psi}_\lambda^+)^0\setminus\{0\})$,  all $t \in [0,1]$ (note $K^a_{\widehat{\psi}_\lambda^+}=\{u_\lambda^\ast\}$),}\\
		\label{eq50} & \widehat{\psi}_\lambda^+(h_0(t,u))\leq \widehat{\psi}_\lambda^+(h_0(s,u))\quad \mbox{for all $0\leq s\leq t\leq 1$, all $u\in ((\widehat{\psi}_\lambda^+)^0\setminus\{0\})$.}
	\end{align}
	
	These properties of the deformation $h_0$ imply that $K^a_{\widehat{\psi}_\lambda^+}$ is a strong deformation retract of $(\widehat{\psi}_\lambda^+)^0\setminus\{0\}$ and the deformation is $\widehat{\psi}_\lambda^+$-decreasing. We set $\gamma_+(t)=h_0(t,\xi \, \widehat{u}_1(q,a_2))^+$ for all $t \in [0,1]$, i.e., a continuous path in $W_0^{1,p}(\Omega)$ and its trace is in the positive cone of $W_0^{1,p}(\Omega)$. Note $\xi \,\widehat{u}_1(q,a_2) \in (\widehat{\psi}_\lambda^+)^0$ (see \eqref{eq47}) and $\widehat{\psi}_\lambda^+(\xi\, \widehat{u}_1(q,a_2))=\widehat{\psi}_\lambda(\xi\, \widehat{u}_1(q,a_2))$. So, we have
	\begin{align}
		\nonumber & \gamma_+(0)=  \xi \,\widehat{u}_1(q,a_2)\quad \mbox{(see \eqref{eq48})},\\
		\nonumber & \gamma_+(1)=u_\lambda^\ast \quad \mbox{(see \eqref{eq49})},\\
		\nonumber & \widehat{\psi}_\lambda^+(\gamma_+(t))\leq \widehat{\psi}_\lambda^+(\gamma_+(0)) \quad \mbox{for all $t\in [0,1]$ (see \eqref{eq50})},\\ \nonumber \Rightarrow \quad & \widehat{\psi}_\lambda(\gamma_+(t))\leq \widehat{\psi}_\lambda(\xi \,\widehat{u}_1(q,a_2))<0 \quad \mbox{for all $t\in [0,1]$ (see \eqref{eq37}, \eqref{eq47})},\\ \label{eq51} \Rightarrow \quad & \widehat{\psi}_\lambda\Big|_{\gamma_+}<0,
	\end{align}
	with $\gamma_+$ being a continuous path in $W_0^{1,p}(\Omega)$, linking $\xi \, \widehat{u}_1(q,a_2)$ to $u_\lambda^\ast$. For $\widehat{\psi}_\lambda^-$, we can produce in a similar way a continuous path $\gamma_-$ in $W_0^{1,p}(\Omega)$, connecting $-\xi \, \widehat{u}_1(q,a_2)$ and $v_\lambda^\ast$.  and such that
	\begin{equation}
		\label{eq52} \widehat{\psi}_\lambda\Big|_{\gamma_-}<0.
	\end{equation}
	
	Merging $\gamma_-$, $\gamma_0$, $\gamma_+$, we get $\gamma_\ast \in \Gamma$ satisfying
	\begin{align*}
		& \widehat{\psi}_\lambda\Big|_{\gamma_\ast}<0 \quad \mbox{(see \eqref{eq47}, \eqref{eq51}, \eqref{eq52}),}\\ \Rightarrow \quad &
		\widehat{\psi}_\lambda(y_\lambda)<0=\widehat{\psi}_\lambda(0),\end{align*}
	which implies $y_\lambda \neq 0$, and so $y_\lambda \in [v_\lambda^\ast,u_\lambda^\ast]\cap C_0^1(\overline{\Omega})$  is nodal solution to \eqref{eq0}.
	
\end{proof}

So, we have the following multiplicity result of \eqref{eq0}. We emphasize that in this theorem, one has sign information for all the solutions and the solutions are ordered.

\begin{theorem}
	\label{th10} Let $H_0$, $H_1$ be satisfied. Thus: 
	\begin{itemize}
		\item[$(a)$] if $\lambda >\widehat{\lambda}_1(q,a_2)$, then \eqref{eq0}  admits at least two constant sign solutions $u_\lambda \in {\rm int \,} C_+$, $v_\lambda \in -{\rm int \,} C_+$;
		\item[$(b)$] if $\lambda >\widehat{\lambda}_2(q,a_2)$, then there is also a nodal solution of \eqref{eq0}, namely $y_\lambda \in [v_\lambda,u_\lambda] \cap C_0^1(\overline{\Omega})$.
	\end{itemize} 
\end{theorem}

If $q=2$ (weighted $(p,2)$-equation), then we can improve a little Theorem \ref{th10}$(b)$. 

\begin{theorem}
	\label{th11} Let $H_0$, $H_1^\prime$ (with $q=2$) be satisfied, and $\lambda >\widehat{\lambda}_2(2,a_2)$. Then, \eqref{eq0}  (with $q=2$) admits at least three nontrivial smooth solutions with sign information and ordered $u_\lambda \in {\rm int \,} C_+$, $v_\lambda \in -{\rm int \,} C_+$, $y_\lambda \in {\rm int \,}_{C_0^1(\overline{\Omega})}[v_\lambda,u_\lambda]$. 
\end{theorem}

\begin{proof}
	We start from the solutions provided by Theorem \ref{th10}, namely $u_\lambda \in {\rm int \,} C_+$, $v_\lambda \in -{\rm int \,} C_+$ and $y_\lambda \in  [v_\lambda,u_\lambda] \cap C_0^1(\overline{\Omega})$ nodal. 
	
	Let $a(z,y)=a_1(z)|y|^{p-2}+a_2(z)y$ for all $z \in \Omega$, all $y \in \mathbb{R}^N$. Thus
	${\rm div \, }a(z,\nabla u)=\Delta_p^{a_1}u+\Delta^{a_2}u$ for all $u \in W_0^{1,p}(\Omega)$. Observe $a(z,\cdot)\in C^1(\mathbb{R}^N,\mathbb{R}^N)$ (recall that $2<p$ here) and 
	\begin{align*}
		& \nabla_y a(z,y)=a_1(z)|y|^{p-2}\left[{\rm id \,}+ (p-2)\frac{y \otimes y}{|y|^2}\right]+a_2(z)	\,{\rm id }\\  \Rightarrow \quad & (\nabla_y a(z,y)\xi,\xi)\geq c_1|\xi|^2
		\quad \mbox{for all $y,\xi \in \mathbb{R}^N$.}
	\end{align*}
	
	Also, if $\rho=\max\{\|v_\lambda\|_\infty, \|u_\lambda\|_\infty\}$ and $\widehat{\xi}_\rho>0$ is taken from $H_1^\prime(vi)$, then
	$$f(z,x)-f(z,u)\geq -\widehat{\xi}_\rho|x-u| \quad \mbox{for all $x,u \in [-\rho,\rho]$.}$$
	
	The tangency principle  (Pucci and Serrin \cite[Theorem 2.5.2]{Ref21}) leads to
	\begin{equation}
		\label{eq53}v_\lambda(z)<y_\lambda(z)<u_\lambda(z)\quad \mbox{for all $z \in \Omega$.}
	\end{equation}
	
	Then we have
	\begin{align}
		-\Delta_p^{a_1}y_\lambda -\Delta^{a_2}y_\lambda+\widehat{\xi}_\rho|y_\lambda|^{p-2}y_\lambda \nonumber & = \lambda y_\lambda+ f(z,y_\lambda)+\widehat{\xi}_\rho |y_\lambda|^{p-2}y_\lambda\\ \nonumber & \leq\lambda  u_\lambda+f(z,u_\lambda)+ \widehat{\xi}_\rho u_\lambda^{p-1}\quad \mbox{(see \eqref{eq53} and   $H_1^\prime(vi)$)}\\ \label{eq54}
		& = -\Delta_p^{a_1}u_\lambda-\Delta^{a_2}u_\lambda+\widehat{\xi}_\rho u_\lambda^{p-1}.
	\end{align}
	
	On account of \eqref{eq53} we have
	$0 \preceq \lambda [u_\lambda-y_\lambda].$
	Returning to \eqref{eq54}, we obtain
	$u_\lambda-y_\lambda\in {\rm int \,}C_+$ (by Gasi\'{n}ski et al. \cite[Proposition 3.2]{Ref6}). On the other side, one can establish that $y_\lambda-v_\lambda\in {\rm int \,}C_+.$ We deduce that  $y_\lambda \in {\rm int \,}_{C_0^1(\overline{\Omega})}[v_\lambda,u_\lambda].$
\end{proof}	

Finally under assumption $H_1^{\prime\prime}$ we can have a nonexistence result.

\begin{theorem}
	\label{th12} Let $H_0$, $H_1^{\prime\prime}$ be satisfied,  and $\lambda < \widehat{\lambda}_1(q,a_2)$. Then,  \eqref{eq0} does not admit nontrivial solution. 
\end{theorem}

\begin{proof}
	At the beginning we postulate the existence of $u \in  \mathcal{S}_\lambda^+ \subseteq {\rm int \,}C_+$ so that
	$$\langle A_{p}^{a_1}(u),h \rangle + \langle A_{q}^{a_2}(u),h \rangle = \int_\Omega [\lambda |u|^{q-2}u+ f(z,u)]h dz \quad \mbox{for all }  h \in W_0^{1,p}(\Omega).
	$$
	
	For $h=u \in W_0^{1,p}(\Omega)$, by $H_1^{\prime\prime}(vi)$ we deduce that
	$$ \int_\Omega a_1(z)|\nabla  u|^pdz-\widehat{\lambda}_1(p,a_1)\|u\|^p_p+\int_\Omega a_2(z)|\nabla u|^qdz-\lambda\|u\|^q_q\leq 0,$$ which implies $[ \widehat{\lambda}_1(q,a_2) -\lambda] \|u\|^q_q\leq 0$,
	a contradiction since $\lambda<\widehat{\lambda}_1(q,a_2)$. Therefore $\mathcal{S}_\lambda^+ =\emptyset$ for all $\lambda<\widehat{\lambda}_1(q,a_2)$.
\end{proof}

\begin{remark}
	For $(p,q)$-equations with no weights but with variable exponents we refer to the survey paper of R\u{a}dulescu \cite{Ref22}.
\end{remark}

\section*{Acknowledgments}
The authors would like to thank Nikolaos S. Papageorgiou
for proposing the problems and providing important comments and suggestions. 
The first author was supported by Slovenian Research Agency
grants P1-0292, N1-0114, and N1-0083.

\end{document}